\documentclass[preprint,11pt]{elsarticle}

\usepackage{amssymb}
\usepackage{enumerate}
\usepackage{amsmath}
\usepackage{amsthm}
\usepackage{hyperref}
\hypersetup{
    colorlinks=true,
    linkcolor=blue,
    filecolor=magenta,
    urlcolor=cyan,
    pdftitle={Overleaf Example},
    pdfpagemode=FullScreen,
    }

\usepackage[margin=0.7in]{geometry}
\usepackage{enumitem}
\newtheorem{theorem}{Theorem}%  meant for continuous numbers
\newtheorem{lemma}[theorem]{Lemma}
\newtheorem{cor}[theorem]{Corollary}
\newtheorem{prop}[theorem]{Proposition}
\newtheorem{proposition}[theorem]{Proposition}%
\newtheorem{remark}{Remark}%
\newtheorem{definition}{Definition}%

\usepackage[utf8]{inputenc}

\usepackage{amssymb,amsmath}
\usepackage{amsbsy}
\usepackage{enumerate}
\usepackage[T1]{fontenc}

\begin{document}

\begin{frontmatter}

%% Title, authors and addresses

%% use the tnoteref command within \title for footnotes;
%% use the tnotetext command for theassociated footnote;
%% use the fnref command within \author or \affiliation for footnotes;
%% use the fntext command for theassociated footnote;
%% use the corref command within \author for corresponding author footnotes;
%% use the cortext command for theassociated footnote;
%% use the ead command for the email address,
%% and the form \ead[url] for the home page:
%% \title{Title\tnoteref{label1}}
%% \tnotetext[label1]{}
%% \author{Name\corref{cor1}\fnref{label2}}
%% \ead{email address}
%% \ead[url]{home page}
%% \fntext[label2]{}
%% \cortext[cor1]{}
%% \affiliation{organization={},
%%             addressline={},
%%             city={},
%%             postcode={},
%%             state={},
%%             country={}}
%% \fntext[label3]{}

\title{Global and nonglobal solutions for a mixed local-nonlocal heat equation}

%% use optional labels to link authors explicitly to addresses:
%% \author[label1,label2]{}
%% \affiliation[label1]{organization={},
%%             addressline={},
%%             city={},
%%             postcode={},
%%             state={},
%%             country={}}
%%
%% \affiliation[label2]{organization={},
%%             addressline={},
%%             city={},
%%             postcode={},
%%             state={},
%%             country={}}

%\author[1]{Alex Lira} %% Author name

\author[1]{Brandon Carhuas-Torre}
\ead{brandon.carhuas@ufpe.br}

\author[2]{Ricardo Castillo}
\ead{rcastillo@ubiobio.cl}

\author[3]{Ricardo Freire}
\ead{ricardo.silva@uesb.edu.br}

\author[2]{Alex Lira}
\ead{aslira@ubiobio.cl}

\author[1]{Miguel Loayza}
\ead{miguel.loayza@ufpe.br}

%% Author affiliation

%\author[2]{Brandon Carhuas}

\affiliation[1]{organization={Departamento de Matem\'atica, Universidade Federal de Pernambuco - UFPE},Department and Organization
          addressline={Av. Prof. Moraes Rego, 1235-Cidade Universitária}, 
           city={Recife}, 
          state={Pernambuco},
           country={Brasil}}

\affiliation[2]{organization={Facultad de Ciencias, Departamento de Matemática, Universidad del B\'io-B\'io - UBB},%Department and Organization
            addressline={Avenida Collao 1202, Casilla 5-C}, 
            city={Concepción},
            state={Bío-Bío},
            country={Chile}}

\affiliation[3]{organization={ Departamento de Ciências Exatas e Tecnológicas, Universidade Estadual do Sudoeste da Bahia - UESB},%Department and Organization 
            addressline={Estr. Bem Querer, Km 04},
            city={Vitória da Conquista}, 
            state={Bahia},
            country={Brasil}}

%% Abstract
\begin{abstract}
%% Text of abstract
In this work, we establish optimal conditions concerning the global and nonglobal existence of a semilinear parabolic equation with a general source term $f$, governed by the mixed local–nonlocal operator $\mathcal{L} = -\Delta + (-\Delta)^s$, where $0 < s < 1$. Furthermore, our findings recover the Fujita exponent $1 + 2s/N$, recently derived by Biagi, Punzo, and Vecchi in \cite{Biagi1}, as well as by Del Pezzo and Ferreira in \cite{Pezzo1}.

\end{abstract}

%%Graphical abstract
%\begin{graphicalabstract}
%\includegraphics{grabs}
%\end{graphicalabstract}

%%Research highlights
%\begin{highlights}
%\item Research highlight 1
%\item Research highlight 2
%\end{highlights}

%% Keywords

\end{frontmatter}

%% Add \usepackage{lineno} before \begin{document} and uncomment 
%% following line to enable line numbers
%% \linenumbers

%% main text
%%

%% Use \section commands to start a section
\section{Introduction}
The main objective of this work is to determine global and nonglobal existence conditions for the following parabolic problem with a mixed local-nonlocal operator
\begin{equation}\label{Eqgeral-1}
	\left\{ 
	\begin{array}{rll}
		u_t + \mathcal{L}u &= h(t) f(u)& \mbox{ in } \mathbb{R}^N \times (0,T), \\	
		{u}(0) &=  {u}_{0}& \mbox{ in } \mathbb{R}^N, \\
	\end{array}
	\right.
\end{equation}
where $f \in C([0,\infty))$ is locally Lipschitz continuous satisfying $f(0)=0$, $h \in C([0,\infty))$, $u_0 \geq 0$, and $T>0$; here $C([0,\infty))$ denotes the set of nonnegative continuous functions defined on the interval $[0,\infty)$. The mixed local-nonlocal operator $\mathcal{L} = -\Delta + (-\Delta)^s$ with  $0<s<1$, appears in various stochastic processes that involve two distinct mechanisms operating at different scales, such as a classical random walk and a Lévy flight (see \cite{Dipierro1}). Additionally, models featuring mixed local-nonlocal operators are utilized across various branches of applied sciences. These applications include optimal animal foraging strategies (see \cite{Dipierro2} and \cite{Dipierro1}), the mitigation of pandemic spread (see \cite{Epstein1}), and heat transport in magnetized plasmas (see \cite{Blazevski1}). Recently, Biagi, Punzo, and Vecchi in \cite{Biagi1} (considering very weak solutions), and later Del Pezzo and Ferreira in \cite{Pezzo1} (within the framework of mild solutions), studied Fujita-type results for the problem \eqref{Eqgeral-1} with $f(u) = u^p$ (for related advancements, see \cite{Kirane1} and \cite{Kumar1}). The authors derived the Fujita exponent, which is given by $1 + 2s/N$. Their findings on blow-up are based on Kaplan's eigenvalue method (see \cite{Kaplan}). Meanwhile, the global existence result presented in \cite{Biagi1} employs an approximation technique that has been utilized in several research papers, such as \cite{CastilloHardy}, \cite{CastiApp}, and \cite{Miguel1}, and \cite{Pascucci}. Additionally, the authors in \cite{Pezzo1} employ a suitable global supersolution.

Motivated by the abovementioned findings, this article establishes optimal conditions for the global and nonglobal solutions for the heat equation \eqref{Eqgeral-1} that involves a mixed local-nonlocal operator. Unlike previous research focused on demonstrating global existence, we utilize a global supersolution of the form  $(1+\gamma) \, e^{-t\mathcal{L}}u_0$, where $\gamma > 0$ is a suitably chosen constant. Here, $\{ e^{-t\mathcal{L}} \}_{t \geq 0}$ represents the heat semigroup associated with \(-\mathcal{L}\) (as referenced in \cite{Pezzo1}). For our results on nonglobal solutions, we employ a new lower bound estimate (see Lemma \ref{Lower1}) alongside techniques introduced by Weissler in \cite{Weissler1, Weissler2} and by Fujita in \cite{Fujita1}. As a result, we recover the Fujita exponent previously established in \cite{Biagi1} and \cite{Pezzo1}.

The remainder of this paper is organized as follows: In Section \ref{prel-res}, we provide the necessary preliminaries. In Section \ref{main-res}, we present and prove our main theorem (Theorem \ref{teo1}). We also apply this theorem to demonstrate Fujita-type results by examining the cases: $f(u) = u^p$, $f(u) = (1 + u)[\ln(u + 1)]^p$, and $f(u) = [\ln(u + 1)]^p/(1 + u)$  (see Corollary \ref{corap} and Corollary \ref{corap2}). Throughout the paper, $C$ denotes a positive constant that may change, even within the same line, without affecting the results.
%Throughout the paper, $C$ represents an arbitrary positive constant, which may vary from line to line and even within the same line. However, its variation does not affect the validity of the results.
%% Labels are used to cross-reference an item using \ref command.
%Section text. See Subsection \ref{subsec1}.
%% Use \subsection commands to start a subsection.
\section{Preliminaries results}\label{prel-res}
According to the results obtained in \cite{Biagi1} and \cite{Song1}, we have the following representation for the semigroup:
\begin{equation} \label{EQ1}
[e^{-t\mathcal{L}}u_0](x) = \int_{\mathbb{R}^N} \mathfrak{p}_t(x - y) u_0(y) \, dy
\end{equation}
for $u_0 \in \mathcal{X}:=L^\infty(\mathbb{R}^N) \cap L^1(\mathbb{R}^N) (\subset L^2(\mathbb{R}^N))$, where
\begin{equation} \label{EQ2}
  \mathfrak{p}_t(z) = \frac{1}{(4\pi t)^{N/2}} \int_{\mathbb{R}^N} e^{- \frac{\lvert z - \xi \rvert^2}{ 4t}} \mathfrak{h}^{(s)}_t(\xi) \, d\xi \quad \left( z \in \mathbb{R}^N,\ t > 0 \right) 
\end{equation}
is the heat kernel of $-\mathcal{L}$. Here $\mathfrak{h}^{(s)}_t$ is the heat kernel of the fractional Laplacian $-(- \Delta)^s$, which satisfies:
\begin{equation} \label{EQ3}
C^{-1}_0 \min \left\{ t^{-N/(2s)}, \frac{t}{|x|^{N+2s}} \right\}
\leq \mathfrak{h}^{(s)}_t(x)
\leq C_0 \min \left\{ t^{-N/(2s)}, \frac{t}{|x|^{N+2s}} \right\}, \, \, \mbox{ for } x \in \mathbb{R}^N \mbox{ and } t > 0,
\end{equation}
for some constant $C_0 \geq 1$. The heat kernel $\mathfrak{p}$ has the following properties (see \cite[Theorem 2.4]{Biagi1}).
\begin{itemize}
    \item[($A_1$)] $0<\mathfrak{p} \in C^\infty(\mathbb{R}^+ \times \mathbb{R}^N)$, and for every $x \in \mathbb{R}^N$ and $t>0$ we have $\mathfrak{p}_t(x)= \mathfrak{p}_t(-x)$ and $\int_{\mathbb{R}^N} \mathfrak{p}_t(x-y)dy=1.$ 
\item[($A_2$)] For every fixed $x \in \mathbb{R}^N$ and $t,\tau > 0$, we have $\int_{\mathbb{R}^N} \mathfrak{p}_t(x-y) \mathfrak{p}_\tau(y)dy=\mathfrak{p}_{t+\tau}(x).$ Moreover, 
\begin{equation} \label{EQ4}
0< \mathfrak{p}_{t}(x) \leq C t^{-\frac{N}{2s}} \ \ \ \ \mbox{ for every } x \in \mathbb{R}^N, \, t>0.
\end{equation}
\end{itemize}
\begin{remark} \label{Upper1} From conditions ($A_1$), ($A_2$), Fubini's theorem, (\ref{EQ1}) and (\ref{EQ4}) we obtain the estimates $\| e^{-t\mathcal{L}}u_0\|_\infty \leq \|u_0\|_\infty$, $\| e^{-t\mathcal{L}}u_0\|_{L^1} \leq \|u_0\|_{L^1}$ and $\| e^{-t\mathcal{L}}u_0\|_{\infty} \leq  C t^{-\frac{N}{2s}}$ for all $u_0 \in \mathcal{X}$, and $t>0$.
\end{remark}
We will use the following notion of mild solution.
\begin{definition} \label{DEF1} Let $h,f \in C([0,\infty))$. Given $0 \leq u_0 \in \mathcal{X}$, a function $u:\mathbb{R}^N \times [0,T) \rightarrow [0,\infty)$ such that $u \in L^\infty((0,T); \mathcal{X})$ is said to be a mild solution of problem \eqref{Eqgeral-1} if
$$u(t) =  e^{-t\mathcal{L}}u_0
+\int_{0}^{t} e^{-(t-\sigma)\mathcal{L}} \, h(\sigma)\,f(u(\sigma))\, d\sigma \,  \mbox{ for } 0 < t < T.$$
When $T = +\infty$, $u$ is called a global mild solution of \eqref{Eqgeral-1}.
\end{definition}
The following proposition can be demonstrated as stated in \cite[Theorem 3.2]{Pezzo1}.
\begin{prop}
 For each $u_0 \in \mathcal{X}$, there exist a $T>0$ and a unique mild solution \( u \in L^{\infty}((0, T); \mathcal{X}) \) to \eqref{Eqgeral-1}.   
\end{prop}
%%%%%%%%%%%%%%%%%%%%%%%%%%%%%%%%%%%%%%%%%%%%%%%%%%%%%%%%%%%%%%%%%%%%%%%%%%%%%%%%%%%%%%%%%%%%%%%%%%%%%%%%%%%%
\begin{definition}\label{Supersolution}
Let $0\leq u_0 \in \mathcal{X}$ and $h,f \in C([0,\infty))$. A function $\overline{w}:\mathbb{R}^N \times [0,T) \rightarrow [0,\infty)$ such that $\overline w \in L^\infty ((0,T); \mathcal{X})$ is called a supersolution of problem \eqref{Eqgeral-1} if 
$$ \mathcal{F}[\overline{w}](t) := e^{-t\mathcal{L}}u_0+\int_{0}^{t} e^{-(t-\sigma)\mathcal{L}} \, h(\sigma)\,f(\overline w(\sigma))\, d\sigma \leq \overline w(t), \, \, \mbox{ for all } t \in (0,T). $$
When $T = +\infty$, $\overline{w}$ is referred to as a global supersolution of \eqref{Eqgeral-1}. 
\end{definition}
The following lemma establishes that finding a supersolution (or a global supersolution) to problem \eqref{Eqgeral-1} is sufficient to guarantee the existence of a solution (or a global solution). Its proof follows the ideas presented by Robinson et al. in \cite[Theorem 1]{Robinson2} and Li in \cite[Lemma 3]{Li}.

\begin{lemma}\label{L2}  Let $0 \leq u_0 \in \mathcal{X}$, $h, f \in C([0,\infty))$, and suppose that $f$ is  nondecreasing. Then, problem \eqref{Eqgeral-1} admits a nontrivial mild solution in $\mathbb{R}^N  \times [0,T)$ if and only if it admits a supersolution  in $\mathbb{R}^N \times [0,T)$.
\end{lemma}

%%%%%%%%%%%%%%%%%%%%%%%%%%%%%%%%%%%%%%%%%%%%%%%%%%%%%%%%%%%%%%%%%%%%%%%%%%%%%%%%%%%%%%%%%%%%%%%%%%%%%%%%%%%%

%\label{sec2}

%Subsection text.

%% Use \subsubsection, \paragraph, \subparagraph commands to 
%% start 3rd, 4th and 5th level sections.
%% Refer following link for more details.
%% https://en.wikibooks.org/wiki/LaTeX/Document_Structure#Sectioning_commands

%\subsubsection{Estimación inferior para $P(t)$}
We will use the following lemmas and proposition to establish our results on nonglobal solutions.
\begin{lemma}[Lower bound estimate for $\mathfrak{p}_t$] \label{Lower1} 
There exists a positive constant $C^{*}$ such that
\begin{equation*}
\mathfrak{p}_t(z)\geq C^{*}t^{ -\frac{N}{2s}}, \, \, \,   \mbox{ for } t>1, \, |z|\leq t^{\frac{1}{2}}, \mbox{ and } 0<s<1. 
\end{equation*}
\end{lemma}
\begin{proof} In what follows, $\omega_{N}$ denotes the volume of the unit ball $B_1(0)$.   From \eqref{EQ2} and \eqref{EQ3}  we have
\begin{equation*} 
\begin{aligned}
\mathfrak{p}_t(z) = \frac{1}{(4\pi t)^{\frac{N}{2}}} \int_{\mathbb{R}^N} e^{-\frac{|z - \xi|^2}{4t}}\mathfrak{h}^{(s)}_t(\xi) \, d\xi \,  &\geq \frac{1}{(4\pi t)^{\frac{N}{2}}} \int_{\mathbb{R}^N} e^{-\frac{|z - \xi|^2}{4t}} C_0^{-1} \min\Big\{ t^{-N/(2s)}, \frac{t}{|\xi|^{N+2s}} \Big\} \, d\xi \\
	&  \geq \frac{1}{(4\pi t)^{\frac{N}{2}}} \int_{|\xi| \leq t^{\frac{1}{2s}}} e^{-\frac{|z - \xi|^2}{4t}} C_0^{-1} \min\Big\{ t^{-N/(2s)}, \frac{t}{|\xi|^{N+2s}} \Big\} \, d\xi\\ 
	&  \geq \frac{t^{-\frac{N}{2s}}}{(4\pi t)^{\frac{N}{2}}}C_0^{-1} \int_{|\xi| < t^{\frac{1}{2}}} e^{-2\left(\frac{|z|^2}{4t}+\frac{|\xi|^2}{4t}\right)} \, d\xi\\
	&  \geq \frac{t^{-\frac{N}{2s}}}{(4\pi t)^{\frac{N}{2}}}C_0^{-1}e^{-1} \int_{|\xi| < t^{\frac{1}{2}}} \, d\xi \geq \frac{t^{-\frac{N}{2s}}}{(4\pi)^{\frac{N}{2}}}C_0^{-1}e^{-1}\omega_{N}
= C^{*} t^{-\frac{N}{2s}}.
\end{aligned}
\end{equation*}
\end{proof}
\begin{lemma}\label{Lower2}
   Let $0 \leq u_0\in \mathcal{X}$, and $u_0\neq 0$. Then there exists a positive constant $C^{*}$ such that 
\begin{equation*}
\left[e^{-t\mathcal{L}}u_0\right](x) \geq C^{*}t^{-\frac{N}{2s}}  \int_{|y|\leq \frac{1}{2}t^{\frac{1}{2}}} u_0(y) \, dy, \ \ \mbox{ for all } t>1 \mbox{ and } |x|< \frac{1}{2} t^{\frac{1}{2}}.
\end{equation*}
\end{lemma}
\begin{proof} From \eqref{EQ1} and noting that $|x - y| \leq t^{\frac{1}{2}}$ for $|y|\leq \frac{1}{2}t^{\frac{1}{2}} $, the result follows directly from Lemma \ref{Lower1}.
 %we obtain $$\left[e^{-t\mathcal{L}}u_0\right](x)  \geq \int_{|y|\leq \frac{1}{2} t^{\frac{1}{2}}} \mathfrak{p}_t(x-y) u_0(y)\, dy \geq  C^{*} t^{-\frac{N}{2s}} \int_{|y|\leq \frac{1}{2} t^{\frac{1}{2}}} u_0(y)\,dy.$$
\end{proof}

\begin{proposition}\label{Prop1} Assume $h \in C([0,\infty))$, $f\in C([0,\infty))$ is nondecreasing such that $f(\sigma)>0$ and $f(e^{-\sigma \mathcal{L}}w_0)\leq e^{-\sigma\mathcal{L}}f(w_0)$ for all $\sigma>0$ and $0\leq w_0 \in \mathcal{X}$, and 
\begin{equation}\label{Osgood}
\int_{z_0}^\infty \frac{d \sigma}{f(\sigma)} < \infty, \, \mbox{ for all } z_0>0.
\end{equation}
Let $u: \mathbb{R}^N \times [0,T) \rightarrow [0,\infty) $ a mild solution of \eqref{Eqgeral-1} with initial condition $0\leq u_0 \in \mathcal{X}$, $u_0 \neq 0$.
Then
\begin{equation*}
\int_0^\tau h(\sigma) d\sigma \cdot \left(\int_{\|e^{-\tau \mathcal{L}}u_0\|_{\infty}}^\infty \frac{d \sigma}{f(\sigma)}\right)^{-1} \leq 1, \, \, \mbox{ for all } \tau \in (0,T).
\end{equation*}
\end{proposition}
\begin{proof} From Definition \ref{DEF1} and \eqref{EQ2}, we have
$$u(t)=e^{-t\mathcal{L}}u_0+\int_{0}^{t} e^{-(t-\sigma)\mathcal{L}}\, h(\sigma)\,f( u(\sigma))\, d\sigma = \mathcal{F}[u](t),  \, \mbox{ for all } \, t \in (0,T).$$
Let $0<t<s$. From $f(e^{-\sigma\mathcal{L}}w_0)\leq e^{-\sigma\mathcal{L}}f(w_0)$ (for all $0\leq w_0 \in \mathcal{X}$), and semigroup properties,  we have 
\begin{equation}\label{eqfuntion}
     e^{-(s-t)\mathcal{L}}u(t) = e^{-(s-t) \mathcal{L}}\mathcal{F}[u](t)\geq \Psi(t) :=e^{-s\mathcal{L}}u_0+\int_{0}^{t}h(\sigma)\,f(e^{-(s-\sigma)\mathcal{L}}u(\sigma))\, d\sigma.
\end{equation}
Note that $\Psi(t)$ is absolutely continuous on $[0,s]$ (with $x$ fixed), which implies that it is differentiable almost everywhere. Hence, since $f$ is nondecreasing, from (\ref{eqfuntion}) we have $\Psi'(t)= h(t)\,f(e^{-(s-t)\mathcal{L}} u(t)) \geq h(t)\,f(\Psi(t))$. This and \eqref{Osgood} imply that
\begin{align*}
\int_{\Psi(0)}^{\infty} \frac{d\sigma}{f(\sigma)} \geq \int_{\Psi(0)}^{\Psi(s)}\frac{d\sigma}{f(\sigma)}
=\int_{\Psi(0)}^{\infty}\frac{d\sigma}{f(\sigma)}-\int_{\Psi(s)}^{\infty}\frac{d\sigma}{f(\sigma)} =-\int_{0}^{s}\left(\int_{\Psi(t)}^{\infty} \frac{d\sigma}{f(\sigma)} \right)'\, dt =\int_{0}^{s} \frac{\Psi'(t)}{f(\Psi(t))} \, dt \geq \int_{0}^{s} h(t)\,dt,
\end{align*}    
for all $s \in (0,T)$. Thus,  as $\int_{0}^{s} h(t)\,dt$ does not depend on $x$, we obtain $\int_{\|e^{-s \mathcal{L}}u_0\|_{\infty}}^{\infty} \frac{d\sigma}{f(\sigma)} \geq \int_{0}^{s} h(t)\,dt$ for all $s \in (0,T)$. 
\end{proof}
\section{Main results}\label{main-res}
\begin{theorem}\label{teo1} Suppose that $h \in C([0,\infty))$, $f\in C([0,\infty))$ is a locally Lipschitz function such that $f(0)=0$.
\begin{itemize}
    \item[(i)] If $f, f(s)/s$ are nondecreasing in some interval $(0, m]$, and there exists $0 \leq v_0 \in \mathcal{X}$, $v_0 \neq 0$ satisfying $\|v_0\|_\infty \leq m$ and the condition
    \begin{equation} \label{Globalcond}   \int_0^{\infty}h(\sigma) \frac{f(\|e^{-\sigma \mathcal{L}}v_0\|_\infty)}{\|e^{-\sigma \mathcal{L}}v_0\|_\infty} \, d\sigma 
        < 1.
    \end{equation}
    Then there exists $\delta > 0$, such that for $\delta v_0 = u_0$, the solution of (\ref{Eqgeral-1}) is a nontrivial global mild solution.

    \item[(ii)] Let $0 \leq u_0 \in \mathcal{X}$, $u_0 \neq 0$ and suppose that $f$ is nondecreasing function such that $0<f(s)$, for all $s>0$, condition \eqref{Osgood} holds, and that it satisfies $f(e^{-t\mathcal{L}}w_0)\leq e^{-t\mathcal{L}}f(w_0)$ for all $t>0$ and $0\leq w_0 \in \mathcal{X}$. If exists $\kappa>0$ such that
\begin{equation}\label{condteo2} \int_{\|e^{-\kappa \mathcal{L}}u_0\|_\infty}^{\infty} \frac{d\sigma}{f(\sigma)} \leq \int_0^{\kappa} h(\sigma) \, d\sigma,
        \end{equation}
    then the solution $u$ of problem \eqref{Eqgeral-1} with initial condition $u_0$ is nonglobal mild solution.
\end{itemize}
\end{theorem}
\begin{proof} $(i)$ Let  $0<\delta \leq \min \left\{ \frac{1}{\beta+1},\frac{m}{(\beta +1)\|v_0\|_{\infty}} \right\}$,
where $\beta$ is chosen so that $\int_0^{\infty}h(\sigma) \frac{f(\|e^{-\sigma \mathcal{L}}v_0\|_\infty)}{\|e^{-\sigma\mathcal{L}}v_0\|_\infty} \, d\sigma 
        < \frac{\beta}{\beta+1}$ (this is possible due to condition \eqref{Globalcond}). We claim that  \(\overline{w}(t)=(1+\beta)e^{-t \mathcal{L}}u_0\) is a global supersolution of problem \eqref{Eqgeral-1}, with $u_0=\delta v_0$. Indeed, given that $f$  and  $f(s)/s$ are nondecreasing in $(0,m]$, together with Remark \ref{Upper1}, the semigroup properties, and the fact that $ (1+\beta)e^{-t \mathcal{L}}u_0 \leq e^{-t \mathcal{L}}v_0 \leq m$ (see Remark \ref{Upper1}), we conclude that
\begin{align*}
\mathcal{F}[\overline{w}](t) & = e^{-t\mathcal{L}}u_0+\int_{0}^{t} e^{-(t-\sigma)\mathcal{L}}\, [h(\sigma)\,f((1+\beta)e^{-\sigma \mathcal{L}}u_0)]\, d\sigma  \\
& \leq e^{-t\mathcal{L}}u_0+\int_{0}^{t} h(\sigma)\,e^{-(t-\sigma)\mathcal{L}}\,\frac{f((1+\beta)e^{-\sigma \mathcal{L}}u_0)}{(1+\beta)e^{-\sigma \mathcal{L}}u_0}\,[(1+\beta)e^{-\sigma \mathcal{L}}u_0]\, d\sigma \\
& \leq e^{-t\mathcal{L}}u_0+[(1+\beta)e^{-t \mathcal{L}}u_0] \int_{0}^{t} h(\sigma)\frac{f(\|e^{-\sigma \mathcal{L}}v_0\|_\infty)}{\|e^{-\sigma \mathcal{L}}v_0\|_\infty}\,\, d\sigma \\
& \leq e^{-t\mathcal{L}}u_0+[(1+\beta)e^{-t \mathcal{L}}u_0]\,\frac{\beta}{\beta+1} \leq (1+\beta)e^{-t\mathcal{L}}u_0 = \overline{w}(t).
\end{align*}
Thus, the claim holds. Hence, by Lemma \ref{L2} and Remark \ref{Upper1}, problem \eqref{Eqgeral-1} admits a nontrivial global mild solution. \\
$(ii)$ Suppose that $u$ is a nontrivial global mild solution of problem \eqref{Eqgeral-1}. Then by Proposition \ref{Prop1}, $\int_{\|e^{-\kappa \mathcal{L}}u_0\|_\infty}^{\infty} \frac{d\sigma}{f(\sigma)} \geq \int_0^{\kappa} h(\sigma) \, d\sigma$ for all $\kappa \in (0, +\infty)$. This is clearly a contradiction with condition \eqref{condteo2}.
%Suppose that $u$ is a nontrivial global mild solution, and that \eqref{condteo2} holds. However, this leads to a contradiction with Proposition \ref{Prop1}.
\end{proof} 
%%%%%%%%%%%%%%%%%%%%%%%%%%%%%%%%%%%%%%%%%%%%%%%%%%%%%%%%%%%%%%%%%%%%%%%%%%%%%%%%%%%%%%%%%%%%%%%%%%%%%%%%%%%%%%%%%%%%%%%%%%%%
 %As an application of Theorem \ref{teo1}, we establish the following corollaries. 

%  ou $f(u) = [\ln(u + 1)]^p/(1 + u)$.

We now derive Fujita-type results for problem \eqref{Eqgeral-1}, focusing on specific functions $f(u)$. First, we address the cases $f(u) = u^p$ and $f(u) = (1 + u)[\ln(u + 1)]^p$. These functions are clearly convex, nonnegative, and nondecreasing in the interval  $(0, +\infty)$ (and therefore locally Lipschitz continuous). Moreover,
$f(0)=0$, and $f(s)/s$ is also nondecreasing.
Additionally, the property $f(e^{-t\mathcal{L}}w_0)\leq e^{-t\mathcal{L}}f(w_0)$ 
 holds due to the convexity of 
$f$ with
$f(0)=0$, and the fact that the semigroup  $e^{-t\mathcal{L}}$ is a submarkovian operator, as established by property  $(A_1)$ (see \cite[Lemma 3.2 and Theorem 3.4]{haase}). Therefore, we can apply Theorem \ref{teo1} to obtain the following result.
 \begin{cor}\label{corap}
 Suppose either $f(u)= (1+u)[\ln(1+u)]^p$ or $f(u)= u^p$ $(p>1)$, and that there exist constants $C_1, \, C_2>0$ such that
 $h \in C([0,\infty)$ satisfies $C_1 t^\rho \leq h(t) \leq C_2 t^\rho$ for some $\rho>0$ and for $t$ sufficiently large.
 %and \(h \in C[0,\infty)\) with \(h(t)\sim t^\rho\) \((\rho>0)\) for \(t\) large enough.
 \begin{enumerate}[label=(\roman*)]
 \item If $1 < p < 1 + \frac{2s(\rho+1)}{N}$, then problem \eqref{Eqgeral-1} has no nontrivial global mild solutions.
 \item If \(p>1 + \frac{2s(\rho+1)}{N}\)  then  problem \eqref{Eqgeral-1} has a nontrivial global mild solution.  
\end{enumerate}
\end{cor}
\begin{proof} We only treat the case $f(u) = (1+u)[\ln(1+u)]^p$, as the case $f(u) = u^p$ can be proved similarly.\\
%(\textbf{\textit{Case} $1 < p < 1 + \frac{2s(\rho+1)}{N}$})
%Suppose that there exists $u$ a nontrivial global mild solution of \eqref{Eqgeral-1} with initial condition $u_0 \neq 0$.
%From Lemma \ref{Lower2} and since $h(t) \sim t^{\rho}$ we obtain that $\|e^{-t\mathcal{L}}u_0\|_\infty \geq c\,t^{-\frac{N}{2s}}$ and $h(t)\geq C_1t^{\rho}$ for all $t$ large enough. 
$(i)$ Suppose $0 \leq u_0 \in \mathcal{X}$ and $u_0 \neq 0$. From Lemma \ref{Lower2} we obtain that $\|e^{-t\mathcal{L}}u_0\|_\infty \geq c\,t^{-\frac{N}{2s}}$ for all $t$ large enough. Thus, since $1<p< 1 + \frac{2s(\rho+1)}{N}$, we have
\begin{eqnarray*}
    \left[ \int_{\|e^{-t\mathcal{L}}u_0\|_\infty}^\infty \frac{d\sigma}{f(\sigma)} \right]^{-1} \int_0^t h(\sigma) d\sigma &\geq& C_1(p - 1)[\ln(1 + \|e^{-t\mathcal{L}}u_0\|_\infty)]^{p - 1} \cdot \int_{\frac{t}{2}}^{t} \sigma^\rho d\sigma\\
&\geq& C \cdot \left[\ln(1 + c \cdot t^{-\frac{N}{2s}}) \cdot t^{\frac{\rho+1}{p - 1}}\right]^{p - 1}  \geq C \cdot \left[c \cdot t^{-\frac{N}{2s}} \cdot t^{\frac{\rho+1}{p - 1}}\right]^{p - 1} > 1,
\end{eqnarray*}
for $t$ sufficiently large. Hence, condition \eqref{condteo2} is satisfied, and the result follows from Theorem \ref{teo1} $(ii)$.\\
$(ii)$  Let $0 \leq v_0 \in \mathcal{X}$, with $v_0 \neq 0$, which will be chosen sufficiently small. By assumption, there exists $t_1 >0$ such that $h(\sigma) \leq C_2 \sigma^\rho$ for all $\sigma \in [t_1, \infty)$.
%Since $h(\sigma) \sim \sigma^\rho$, there exist $C > 0$ and $t_1 >0$ such that $h(\sigma) \leq C \sigma^\rho$ for all $\sigma \in [t_1, \infty)$. 
Now let $t_0 >t_1$ (which will be chosen sufficiently large). Using the inequality $\ln(1+z)\leq z$ (for $z>0$), the fact that $f(s)/s$ is nondecreasing, and Remark \ref{Upper1}, we obtain
%, equation \eqref{EQ1}, inequality \eqref{EQ4},
\begin{align*}
 \int_0^{\infty}h(\sigma) \frac{f(\|e^{-\sigma \mathcal{L}}v_0\|_\infty)}{\|e^{-\sigma\mathcal{L}}v_0\|_\infty} &\, d\sigma 
    = \int_0^{t_0}h(\sigma) \frac{f(\|e^{-\sigma \mathcal{L}}v_0\|_\infty)}{\|e^{-\sigma\mathcal{L}}v_0\|_\infty} \, d\sigma +
\int_{t_0}^{\infty}h(\sigma) \frac{f(\|e^{-\sigma \mathcal{L}}v_0\|_\infty)}{\|e^{-\sigma\mathcal{L}}v_0\|_\infty} \, d\sigma \\
    &= \frac{f(\| v_0\|_\infty)}{\|v_0\|_\infty} \int_0^{t_0} h(\sigma) \, d\sigma + \int_{t_0}^{\infty}h(\sigma) \frac{(1+\|e^{-\sigma \mathcal{L}}v_0\|_\infty)[\ln(1+\|e^{-\sigma \mathcal{L}}v_0\|_\infty)]^p}{\|e^{-\sigma \mathcal{L}}v_0\|_\infty} \, d\sigma \\
    %&\leq \frac{f(\| v_0\|_\infty)}{\|v_0\|_\infty} \int_0^{t_0} h(\sigma) \, d\sigma + (1+\|v_0\|_\infty) \int_{t_0}^{\infty}h(\sigma) \frac{[\ln(1+\|e^{-\sigma \mathcal{L}}v_0\|_\infty)]^p}{\|e^{-\sigma \mathcal{L}}v_0\|_\infty} \, d\sigma \\
    & \leq \frac{f(\| v_0\|_\infty)}{\|v_0\|_\infty} \int_0^{t_0} h(\sigma) \, d\sigma + (1+\|v_0\|_\infty) \int_{t_0}^{\infty}h(\sigma)\|e^{-\sigma \mathcal{L}}v_0\|_\infty^{p-1} \, d\sigma \\
    & \leq \frac{f(\| v_0\|_\infty)}{\|v_0\|_\infty} \int_0^{t_0} h(\sigma) \, d\sigma + C \int_{t_0}^{\infty}\sigma^{\rho-\frac{N}{2s}(p-1)} \, d\sigma \\
   & = \frac{f(\| v_0\|_\infty)}{\|v_0\|_\infty} \int_0^{t_0} h(\sigma) \, d\sigma  + C \frac{t_0^{\rho+1-\frac{N}{2s}(p-1)}}{\frac{N}{2s}(p-1)-(\rho+1)}.
   %\Big{|}_{t_0}^{\infty}.
\end{align*}
Since $ \frac{f(s)}{s} = \frac{(1+s)[\ln(1+s)]^p}{s} \to 0 $ as $s \to 0^+$, the result follows from Theorem \ref{teo1} $(i)$ by first choosing \( t_0 \) sufficiently large and then \( \|v_0\|_\infty \) sufficiently small such that
$\int_0^{\infty} h(\sigma) \frac{f(\|e^{-\sigma \mathcal{L}}v_0\|_\infty)}{\|e^{-\sigma \mathcal{L}}v_0\|_\infty} \, d\sigma < 1$.
\end{proof}
%%%%%%%%%%%%%%%%%%%%%%%%%%%%%%%%%%%%%%%%%%%%%%%%%%%%%%%%%%%%%%%%%%%%%%%%%%%
Finally, we analyze the case $f(u) = [\ln(u + 1)]^p/(1 + u)$:
 \begin{cor} \label{corap2}
 Suppose that \(f(u)= [\ln(1+u)]^p/(1+u)\) \((p>1)\) and $h=1$.
 \begin{enumerate}[label=(\roman*)]
 \item If $1 < p \leq 1 + \frac{2s}{N}$, then problem \eqref{Eqgeral-1} has no nontrivial global mild solutions.

 \item If $p>1 + \frac{2s}{N}$  then  problem \eqref{Eqgeral-1} has a nontrivial global mild solution.  
\end{enumerate}
\end{cor}
\begin{proof} Noting that in this case both $f(s)$ and 
$f(s)/s$ are nondecreasing on an interval $(0,m)$, $m>0$, item $(ii)$ can be proved in a similar manner to the proof of Corollary \ref{corap} $(ii)$.\\
%(\textbf{\textit{Case} $1 < p \leq 1 + \frac{2s}{N}$}).
$(i)$ Here we cannot directly use Theorem \ref{teo1} since $f$ is not convex, and thus we cannot guarantee that  $f(e^{-t\mathcal{L}}w_0)\leq e^{-t\mathcal{L}}f(w_0)$.  Therefore, we will adopt a proof based on a comparison criterion. Suppose that there exists \(u \in L^{\infty}((0,\infty),\mathcal{X})\) a nontrivial global solution to the problem (\ref{Eqgeral-1}) and let \(M_0 \geq \|u\|_\infty\). From the inequality
$\ln(z+1) \geq z/(z+1) \, (z >0)$, we have
\begin{eqnarray*}
    u(t)  = e^{-t\mathcal{L}}u_0+\int_{0}^{t} e^{-(t-\sigma)\mathcal{L}} \frac{[\ln(1+u(\sigma))]^p}{1+u(\sigma)} d\sigma & \geq& e^{-t\mathcal{L}}u_0+\int_{0}^{t} e^{-(t-\sigma)\mathcal{L}} \frac{u^p(\sigma)}{[1+\|u(\sigma)\|_\infty]^{p+1}} d\sigma
    \\
    &\geq& e^{-t\mathcal{L}}u_0+\int_{0}^{t} e^{-(t-\sigma)\mathcal{L}} \frac{u^p(\sigma)}{[1+M_0]^{p+1}} d\sigma.
\end{eqnarray*}
Thus, $u$ is a global supersolution of problem (see Definition \ref{Supersolution})
\begin{equation}\label{Prolemsupsol}
	\left\{ 
	\begin{array}{rll}
		U_t + \mathcal{L} U &= \dfrac{1}{[1+M_0]^{p+1}} U^p & \mbox{ in } \mathbb{R}^N \times (0,T), \\	
		{U}(0) &=  {U}_{0} >0& \mbox{ in } \mathbb{R}^N. \\
	\end{array}
	\right.
\end{equation}
Then, (\ref{Prolemsupsol}) admits a nontrivial mild solution in $\mathbb{R}^N \times [0,T)$. However, this contradicts the established result in \cite[Theorem 1.3]{Pezzo1}, which states that for $1 < p \leq 1 + 2s/N$ the problem \eqref{Prolemsupsol} has no nontrivial global mild solutions.
%(ii)(\textbf{\textit{Case} $p^\star < p$}) Let $0 \leq v_0 \in \mathcal{X}$, with $v_0 \neq 0$, which will be chosen sufficiently small. From inequality $\ln(1+z)/(1+z)\leq z/(1+z)$ (for $z>-1$), \eqref{EQ1}, \eqref{EQ4}, \eqref{Upper1},  we have
%\begin{align*}    \int_0^{\infty} \frac{f(\|e^{-\sigma \mathcal{L}}v_0\|_\infty)}{\|e^{-\sigma\mathcal{L}}v_0\|_\infty} \, d\sigma 
%    &= \overbrace{\int_0^{t_0}\frac{f(\|e^{-\sigma \mathcal{L}}v_0\|_\infty)}{\|e^{-\sigma\mathcal{L}}v_0\|_\infty} \, d\sigma}^{I(t_0)} +
%\int_{t_0}^{\infty} \frac{f(\|e^{-\sigma \mathcal{L}}v_0\|_\infty)}{\|e^{-\sigma\mathcal{L}}v_0\|_\infty} \, d\sigma \\
%    &= I(t_0) + \int_{t_0}^{\infty} \frac{[\ln(1+\|e^{-\sigma \mathcal{L}}v_0\|_\infty)]^p}{(1+\|e^{-\sigma \mathcal{L}}v_0\|_\infty)\|e^{-\sigma \mathcal{L}}v_0\|_\infty} \, d\sigma \\
 %   &\leq I(t_0) + \int_{t_0}^{\infty} \frac{\|e^{-\sigma \mathcal{L}}v_0\|_\infty^p}{(1+\|e^{-\sigma \mathcal{L}}v_0\|_\infty)\|e^{-\sigma \mathcal{L}}v_0\|_\infty} \, d\sigma \\
%    &\leq I(t_0) + \int_{t_0}^{\infty} \|e^{-\sigma \mathcal{L}}v_0\|_\infty^{p-1} \, d\sigma \\
 %    &\leq I(t_0) + \int_{t_0}^{\infty} \sigma^{-\frac{N}{2s}(p-1)}  \, d\sigma = I(t_0) + \frac{t^{1-\frac{N}{2s}(p-1)}}{1-\frac{N}{2s}(p-1)}\Big{|}_{t_0}^{\infty}
%\end{align*}
%Hence, the result follows from Theorem \ref{teo1} by choosing $t_0$ sufficiently large and $\|v_0\|_\infty$ sufficiently small, such that $\int_0^{\infty}h(\sigma) \frac{f(\|e^{-\sigma \mathcal{L}}v_0\|_\infty)}{\|e^{-\sigma\mathcal{L}}v_0\|_\infty} \, d\sigma <1$.
\end{proof}
~~\\
\textbf{Acknowledgements}\\
R. Castillo is supported by ANID-FONDECYT project No. 11220152. A. Lira is supported by the ANID doctoral scholarship number 21241184. M. Loayza is supported by CNPq - 313382/2023-9.\\
\\
\textbf{Data availability}\\ Data sharing does not apply to this article since no datasets were generated or analyzed during the current study.

\end{document}